\documentclass[smallextended,numbook,runningheads]{svjour3}    
\usepackage{amsmath,amssymb}
\usepackage{color}
\usepackage{graphicx}
\usepackage{lineno,hyperref}
\modulolinenumbers[5]

\newcommand\s{\boldsymbol{\sigma}}
\renewcommand\t{\boldsymbol{\tau}}
\newcommand\RE{\mathbb R}
\newcommand\Ld{L^2(\Omega)}
\newcommand\bLd{\mathbf{L}^2(\Omega)}
\newcommand\Hdiv{\mathbf{H}(\Div;\Omega)}
\DeclareMathOperator*{\Div}{\textrm{div}}
\newcommand\T{\mathcal{T}}
\newcommand\E{\mathcal{E}}
\newcommand\x{\mathbf x}
\newcommand\proj{\mathcal{P}_h}
\newcommand\Upost{U_h^*}
\newcommand\Upostpost{U_h^{**}}
\newcommand\upost{u_h^*}
\newcommand\upostpost{u_h^{**}}
\DeclareMathOperator*{\Grad}{\boldsymbol{\nabla}}
\newcommand*{\jump}[1]{\lbrack\hspace{-2pt}\lbrack%
#1\rbrack\hspace{-2pt}\rbrack}

\newcommand\dofs{N}

\newtheorem{thm}{Theorem}

\numberwithin{equation}{section}

\begin{document}

\title{A posteriori error analysis for the mixed Laplace eigenvalue
problem}

\author{Fleurianne Bertrand      \and Daniele Boffi
       \and    Rolf Stenberg}
        
\institute{Fleurianne Bertrand  \at
           Institut f\"ur Mathematik, Humboldt Universit\"at zu Berlin, Germany\\
\email{fb@math.hu-berlin.de}         
           \and Daniele Boffi 
            \at
            Dipartimento di Matematica ``F. Casorati'', Universit\`a di Pavia,
Italy and Department of Mathematics and System Analysis, Aalto University,
Finland\\
\email{daniele.boffi@unipv.it}
    \and
    Rolf Stenberg \at Department of Mathematics and Systems Analysis,
Aalto University, P.O. Box 11100, 00076 Aalto, Finland\\
    \email{rolf.stenberg@aalto.fi}  
}

\date{Received: date / Accepted: date}

\maketitle

\begin{abstract}
This paper derives a posteriori error estimates for the mixed numerical approximation of the Laplace eigenvalue problem with homogeneous Dirichlet boundary conditions. In particular, the
resulting error estimator constitutes an upper bound for the error and is shown to be local efficient. Therefore, we present a reconstruction in the standard $H_0^1$-conforming space for the primal variable of the mixed Laplace eigenvalue problem. This reconstruction is performed locally on a set of vertex patches. 
\end{abstract}

\section{Introduction}

Eigenvalue problems arise in countless applications and the precise numerical approximation 
of eigenvalues and eigenvectors of elliptic operators is therefor crucial. Finite element methods 
for these problems have been widely used and analyzed under a general framework in many 
works, e.g. in \cite{BabOsb:91} and the references therein. 
The Laplace problem was shown to be an interesting model problem worth to start with.
Convergence and optimal a priori estimates for both eigenvalues and
eigenfunctions are given e.g. in \cite{acta,bbf}. 

It is well-known that on general domains, the eigenfunction can not excepted to be sufficiently regular to use these a priori estimates. Therefore, it is important to use adaptive procedures based on a posteriori error estimators in order to retain the optimal convergence order and several approaches have been considered to construct estimators based 
 on the residual equation presented in \cite{AinOde:93} and \cite{MR3059294}.
In particular 
 \cite{DurPadRod:03}, obtained an a posteriori error estimator for the linear 
finite element approximation of second order elliptic eigenvalue problems using
a residual type error estimator.

For the study of the optimal convergence rates of these procedures, important progress has been made during the last decades. The crucial marking provided in \cite{MR1393904} allows  the error measured in the energy norm to decrease at a constant rate in each step until a prescribed error bound is reached. \cite{MR1770058} showed that the strict error reduction used therfor cannot be expected in general and proved convergence by introducing the concept of data oscillation and the interior node property. \cite{MR2324418} provided a new overall theoretical understanding of adaptive FEM in order to realize optimal computational complexity.
\cite{MR2421046} proved  the optimal cardinality of the AFEM using a decay between consecutive loops.  A survey and an axiomatic presentation of the proof of optimal convergence rates for adaptive finite element methods can be found in \cite{MR2648380,MR3076038,CarFeiPagPra:14}. For non-conforming elements, see \cite{MR2595052}.

An alternative approach for a posteriori error estimation is based on the hypercirle identity 
dating back at least as far as \cite{LadLeg:83} and \cite{PraSyn:47}. In fact, using a flux 
reconstruction of the primal variable of the source problem usually leads to guaranteed, easily,
 fully, and locally computable, upper bound on the error measured in the energy norm, see e.g. 
 \cite{BraSch:08,ErnVoh:15}.   A unified framework for a posteriori error estimation based on 
 stress reconstruction for the Stokes system is carried out in \cite{HanSteVoh:12}, and this was extended to the linear elasticity problem in \cite{BerMolSta:17}.
 
In addition to standard conforming Galerkin approximations, mixed finite elements are very 
popular since they provide a good approximation of the eigenvalues and eigenfunctions, see  \cite{acta,bbf}. A study of optimal rates of convergence was considered in \cite{MR2772091} using a residual based error estimator. For the mixed scheme however, the hypercircle identity can be used as an alternativ approach in order to obtain a considerably better approximation for the scalar variable of the source problem, see \cite{StM2an91}. This paper aims to extend this approach to the eigenvalue.

Mixed finite element methods are an alternative to standard conforming methods. The methods yield an accurate approximation to the flux, whereas that of the scalar variables is less accurate. However, the discrete solution can be used in order to obtain a considerable better approximation to the scalar variable by local post processing schemes, as first shown in \cite{MR813687}. The postprocessing can also be used to obtain an efficient  a posteriori error estimator for the source problem, cf. \cite{MR2240629}. In this paper, we will use the post processing combined with the idea behind the hypercircle method in order to design an a posteriori estimator for the eigenvalue problem. We use the Raviart-Thomas-Nedelec mixed methods since for other main family, i.e. Brezzi-Douglas-Marini,   the approximation properties needed are not valid, cf. Remark \ref{BDMrem} below.

\section{Setting of the problem and preliminary results}

Let $\Omega$ be a polygon in $\RE^2$ or a polyhedron in $\RE^3$.
We consider the standard mixed formulation of the Laplace eigenvalue problem
with homogeneous Dirichlet boundary conditions: find $\lambda\in\RE$ and a non
vanishing $u\in\Ld$, and  $\s\in\Hdiv$, such that it holds%
\begin{equation}
\left\{
\aligned
&(\s,\t)+(\Div\t,u)=0&&\forall\t\in\Hdiv\\
&(\Div\s,v)=-\lambda(u,v)&&\forall v\in\Ld.
\endaligned
\right.
\label{eq:mixed}
\end{equation}

Given a triangulation $\T_h$ of $\Omega$ and an integer $k\ge0$, we define
$\Sigma_h$ as the standard Raviart--Thomas space of order $k$ (sometimes
called Raviart--Thomas-N\'ed\'elec space)
\begin{equation}
\Sigma_h=\{\t\in\Hdiv:\t|_K\in[P_k(K)]^d\oplus\x\tilde P_k(K)\ \forall
K\in\T_h\},
\end{equation}
where $\tilde P_k(K)$ is the space of homogeneous polynomials of degree $k$,
and $U_h$ as the space of discontinuous piecewise polynomials of  $k$
\begin{equation}\label{scalarFE}
U_h=\{v\in\Ld:v|_K\in P_k(K)\ \forall K\in\T_h\}.
\end{equation}
It is well known (see for instance~\cite{acta,bbf} for a review) that the
following scheme provides a good approximation of the eigenvalues and
eigenfunctions of~\eqref{eq:mixed}: find $\lambda_h\in\RE$ and a non
vanishing $u_h\in U_h$, and $\s_h\in\Sigma_h$,  such that it holds
\begin{equation}
\left\{
\aligned
&(\s_h,\t)+(\Div\t,u_h)=0&&\forall\t\in\Sigma_h\\
&(\Div\s_h,v)=-\lambda_h(u_h,v)&&\forall v\in U_h.
\endaligned
\right.
\label{eq:mixedh}
\end{equation}

Let $(\lambda,u,\s)$ be a solution of~\eqref{eq:mixed} corresponding to a
simple eigenvalue. In general there exists $s>1/2$ such that $u\in
H^{1+s}(\Omega)$ and $\s\in\mathbf{H}^s(\Omega)\cap\Hdiv$.
We have that there exist a discrete solution
$(\lambda_h,u_h,\s_h)$ of~\eqref{eq:mixedh} such that
\begin{equation}
\aligned
&|\lambda-\lambda_h|\le Ch^{2r}|u|_{H^{r+1}(\Omega)}\\
&\|u-u_h\|_{\Ld}
+\|\s-\s_h\|_{\mathbf{L}^2(\Omega)}\le Ch^r|u|_{H^{r+1}(\Omega)}\\
&\|\s-\s_h\|_{\mathbf{H}(\Div;\Omega)}\le Ch^r|u|_{H^{r+1}(\Omega)}
\endaligned
\label{eq:errest}
\end{equation}
with $r=\min(s,k+1)$.

\begin{remark}

Estimate~\eqref{eq:errest} requires a careful choice of $u_h$ and $\s_h$.
It is well known that this result can be achieved with an appropriate
normalization and choice of the sign. The case of eigenvalues with
multiplicity higher than one needs more attention $($see, for instance, the
discussion in~\cite[Sec.~3.1]{acta}$)$ and it has been recently observed
$($see~\cite{gallistl}$)$ that an effective a posteriori analysis should consider
clusters of eigenvalues. In~\cite{bggg} a complete analysis of an adaptive
scheme for the problem under consideration has been performed in the framework
of clusters approximation; for the sake of readability, in this paper we
discuss the case of a simple eigenvalue, being understood that our result can
be extended to the more realistic situation using the techniques
of~\cite{gallistl,bggg}.
\label{re:multiple}
\end{remark}

\begin{remark} In  \cite[Lemma 3.2]{MR1722056} the following equality is proved for mixed methods in general
\begin{equation}\label{magic}
\lambda-\lambda_h=\|\s-\s_h\|_{\mathbf{L}^2(\Omega)}^2-\lambda_h\|u-u_h\|_{\Ld}^2.
\end{equation}
\end{remark}
From this it follows that for the BDM method with the same space for the scalar unknown, contrary to what is the case for the source problem, the accuracy is the same as for the RT. Therefore we in this paper do not consider the BDM alternative.

Let $\proj:\Ld\to U_h$ denote the $\Ld$ projection. It has been shown
in~\cite{roberts} that for the source problems associated
with~\eqref{eq:mixed} and~\eqref{eq:mixedh} the term $\|\proj
u-u_h\|_{L^2(\Omega)}$ achieves a higher rate of convergence than the actual
error $\|u-u_h\|_{L^2(\Omega)}$ (see also~\cite[Sec.~7.4]{bbf}). Not
surprisingly, the same result holds true also for the eigenfunctions of
~\eqref{eq:mixed} and~\eqref{eq:mixedh}, even if the proof in the case of the
eigenvalue problem is not a trivial extension of the original one. The
following estimate has been first found in~\cite{gardini} (lowest order case,
using the equivalence with a non-conforming approximation) and then proved
in~\cite[Sec.~6.1]{bggg} (general case, notation suited for clusters of
eigenvalue) and in~\cite[Sec.~3]{rodolfo1} and~\cite[Lemma~10]{rodolfo2}
(these last two results are for a mixed formulation associated with the
Maxwell eigenvalue problem, but the same proof applies to our situation with
the natural modifications): there exists $\rho(h)$, tending to zero as $h$
goes to zero, such that
\begin{equation}
\|\proj u-u_h\|_{\Ld}\le\rho(h)
\left(\|u-u_h\|_{\Ld}+\|\s-\s_h\|_{\bLd}\right).
\label{eq:superconv}
\end{equation}
The value of $\rho(h)$ depends on the regularity of our problem. In our case
it is at least $\rho(h)=O(h^{1/2})$ and, when the domain is convex, we have
$\rho(h)=O(h)$.

We now recall the postprocessing procedure which has been introduced
in~\cite{StM2an91} in order to improve the approximation of $u$ in the case of
the source problem associated with~\eqref{eq:mixed}. Let $\Upost$ be the space
of polynomials of degree up to $k+1$
\[
\Upost=\{v\in\Ld:v|_K\in P_{k+1}(K)\ \forall K\in\T_h\}.
\]
The postprocessed solution $\upost\in\Upost$ is defined such that
\[
\left\{
\aligned
&P_h\upost=u_h\\
&(\Grad\upost,\Grad v)_K=(\s_h,\Grad v)_K\quad\forall v\in(I-P_h)\Upost|_K\
\forall K\in\T_h.
\endaligned
\right.
\]

A proof analogous to the one presented in~\cite{StM2an91} for the source
problem, together with the estimate~\eqref{eq:superconv}, can be applied to
the eigenvalue problem in order to obtain the following bounds for all
$k\ge0$: 
\[
\aligned
 \left(\sum_{K\in\T_h}\|\Grad u-\Grad\upost\|_{L^2(K)}^2\right)^{1/2}
&+\left(\sum_{\ell\in\E_h}h_\ell^{-1}
\|\jump{\upost}\|_{L^2(\ell)}^2\right)^{1/2}
\le Ch^r|u|_{H^{r+1}(\Omega)},\\
 \|u-\upost\|_{\Ld}&\le C\rho(h)h^r|u|_{H^{r+1}(\Omega)},
\endaligned
\]
where $\E_h$ is the set of edges in the triangulation $\T_h$, $h_\ell$ is the
length of $\ell$, $\jump{\cdot}$ denotes the jump, $\rho(h)$ comes
from~\eqref{eq:superconv}, and $r$ has been defined after
Equation~\eqref{eq:errest}. In particular, if the domain is convex, the right
hand side of the last estimate is $O(h^{r+1})$.

\begin{remark}

In~\cite{StM2an91} the last estimate has been proved for $k\ge1$.
However, it was observed that if the source term belongs to the discrete
space $U_h$, then the optimal bound holds true also in the lowest order case
$($see Theorem~2.1, Theorem~2.2, and Remark~2.1$)$. In our case, the right-hand
side $($corresponding to the source term in the framework of~\cite{StM2an91}$)$ is
equal to $\lambda_hu_h\in U_h$, so that the optimal bound is valid also when
$k=0$.

\end{remark}

The next step consists in building a second postprocessed solution which is
globally continuous. This will allow to consider its gradient over the entire
domain $\Omega$. The standard procedure consists in averaging the degrees of
freedom and it is sometimes called Oswald interpolation, being a particular
case of the more general case considered in~\cite{oswald}. For more details
the reader is referred to~\cite[Sec.~5.5.2]{dipietroern}, where the
construction is explained together with the relevant estimates.
Using the Oswald interpolation, starting from $\upost$ we can build an element
$\upostpost$ in
\begin{equation}
\Upostpost=\{v\in H^1_0(\Omega):v|_K\in P_{k+1}(K)\ \forall K\in\T_h\}
\end{equation}
such that
\begin{equation}
\aligned
&\|\Grad u-\Grad\upostpost\|_{\bLd}\le Ch^r|u|_{H^{r+1}(\Omega)}\\
&\|u-\upostpost\|_{\Ld}\le C\rho(h)h^r|u|_{H^{r+1}(\Omega)}.
\endaligned
\label{eq:higher}
\end{equation}
Moreover, the degrees of freedom on the boundary can be set equal to zero so
that $\upostpost$ belongs to $H^1_0(\Omega)$.

\section{A posteriori analysis}

We define the following local error estimator
\[
\eta(K)=\|\Grad\upostpost-\s_h\|_{\mathbf{L}^2(K)}
\]
and the corresponding global error estimator
\[
\eta=\left(\sum_{K\in\T_h}\eta(K)^2\right)^{1/2}=
\|\Grad\upostpost-\s_h\|_{\mathbf{L}^2(\Omega)}.
\]

It is immediate to see that the above error estimator is locally efficient
.

\begin{thm}[Local efficiency]
Let $(\lambda,u,\s)$ be a solution of~\eqref{eq:mixed} corresponding to a
simple eigenvalue $($see Remark~\ref{re:multiple}$)$ and $(\lambda_h,u_h,\s_h)$
the associated discrete solution in the spirit of estimate~\eqref{eq:errest}.
Then the following estimate holds true for each element $K$ in $\T_h$
\[
\eta(K)\le
\sqrt{2}\left(
\|\Grad u-\Grad\upostpost\|_{\mathbf{L}^2(K)}+\|\s-\s_h\|_{\mathbf{L}^2(K)}
\right)^{\frac 1 2}.
\]

\begin{proof}

The proof is immediate from $\Grad u=\s$ and the triangular inequality
\[
\eta(K)=\|\Grad\upostpost-\s_h\|_{\mathbf{L}^2(K)}\le
\|\Grad u-\Grad\upostpost\|_{\mathbf{L}^2(K)}+\|\s-\s_h\|_{\mathbf{L}^2(K)}.
\]

\end{proof}

\end{thm}

The following theorem provides a reliability estimate with constant equal to
one.

\begin{thm}[Global reliability]
\label{Globalreliability}
Let $(\lambda,u,\s)$ be a solution of~\eqref{eq:mixed} corresponding to a
simple eigenvalue $($see Remark~\ref{re:multiple}$)$ and $(\lambda_h,u_h,\s_h)$
the associated discrete solution in the spirit of estimate~\eqref{eq:errest}.
Then the following estimate holds true
\begin{align}
\label{eq:reliability}
\left(
\|\Grad u-\Grad\upostpost\|_{\bLd}^2+\|\s-\s_h\|_{\bLd}^2
\right)^{\frac 1 2}
\le\eta+hot(h),
\end{align}
where $hot(h)$ is a higher order term with respect to $h^r$ as $h$ goes to
zero.

\end{thm}

\begin{proof}

From $\Grad u=\s$ we have
\[
\aligned
\|\Grad\upostpost-\s_h\|_{\bLd}^2&=\|\Grad\upostpost  - \Grad u+\s-\s_h\|_{\bLd}^2\\
&=\|   \Grad\upostpost  - \Grad u  \|_{\bLd}^2+\|\s-\s_h\|_{\bLd}^2\\
&\quad+2(\Grad\upostpost  - \Grad u,\s-\s_h)_{\bLd}\\
&=\|    \Grad\upostpost  - \Grad u  \|_{\bLd}^2+\|\s-\s_h\|_{\bLd}^2\\
&\quad-2(\upostpost  -  u,\Div(\s-\s_h))_{\Ld}.
\endaligned
\]
Since $\Div\s=-\lambda u$ and $\Div\s_h=-\lambda_hu_h$, we obtain
\[
\|\Grad u-\Grad\upostpost\|_{\bLd}^2+\|\s-\s_h\|_{\bLd}^2\le
\eta^2+2\|u-\upostpost\|_{\Ld}\|\lambda u-\lambda_hu_h\|_{\Ld}.
\]
We now conclude the proof by showing that the last term of the previous
equation is a higher order term.

We have already seen in~\eqref{eq:higher} that $\|u-\upostpost\|_{\Ld}$ is of
order $O(\rho(h)h^r)$. The other term can be estimated as follows
\[
\aligned
\|\lambda u-\lambda_hu_h\|_{\Ld}&=
\|\lambda(u-u_h)+(\lambda-\lambda_h)u_h\|_{\Ld}\\
&\le\lambda\|u-u_h\|_{\Ld}+|\lambda-\lambda_h|\|u_h\|_{\Ld}\\
&\le C (h^r+h^{2r})|u|_{H^{r+1}(\Omega)},
\endaligned
\]
so that we finally get
\[
\|\Grad u-\Grad\upostpost\|_{\bLd}^2+\|\s-\s_h\|_{\bLd}^2\le
\eta^2+O(\rho(h)h^{2r}).
\]

\end{proof}

\section{Numerical illustration}

\begin{figure}[ht]\center
\includegraphics[width=\textwidth]{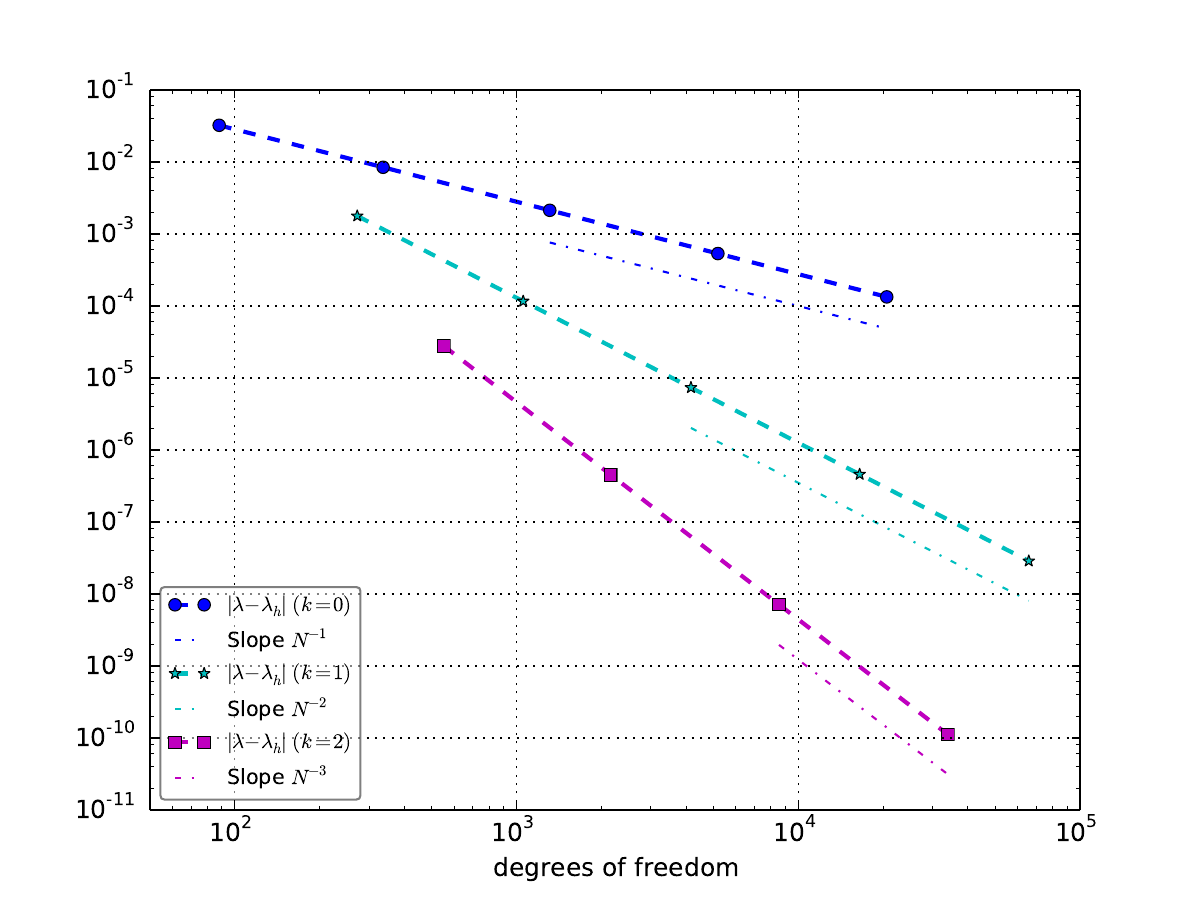}
\caption{Convergence of $\lambda$ (square mesh)}
\label{fig:unieig}
\end{figure}

This section is devoted to numerical results using uniform 
and adaptive $h$-refinements of triangular meshes 
for two-dimensional domains.

The aim of our tests is twofold: on one hand we want to confirm our theory
developed in the previous section; on the other hand we want to investigate
the convergence of the adaptive scheme based on the error indicator that we
have introduced.

We compute the solution of problem~\eqref{eq:mixedh} with Raviart--Thomas
elements with $k=0$ and $k=1$.

The first computations are carried on the square domain $[0,2\pi]^2$ 
with homogeneous Dirichlet boundary conditions imposed on the entire boundary, 
similarly to \cite{acta}. 
In this case, the smallest eigenvalue of \eqref{eq:mixed} is known to be $\lambda=2$.
A corresponding eigenfunction is given by $u(x,y)=\sin x\sin y $. Since the
eigenfunction is smooth, we can take $s=2$ in the error
estimates~\eqref{eq:errest}.
The approximation of the eigenvalue is illustrated in Figure \ref{fig:unieig}.
Since on a uniform mesh the number of degrees of freedom $\dofs$ is
proportional to $h^{-2}$, we can see that the eigenvalues converge with
optimal order $O(h^2)$ and $O(h^4)$ for $k=0$ and $k=1$, respectively. In Figure \ref{fig:unieta}, we see that the error estimator converges with the same rate. The corresponding convergence rates are summarized in Figure \ref{fig:conv}.

In order to investigate the efficiency of the error estimator with respect to
the natural norms, the efficiency index is illustrated
in Figure~\ref{fig:unietac}.

In Figure~\ref{fig:unihot} we check the validity of Theorem~\ref{Globalreliability}: in particular,
we compute the rate of the higher order terms involved with the reliability
estimate. More precisely, the lines in Figure~\ref{fig:unihot} refer to the
difference between the left and right hand sides of equation \eqref{eq:reliability}.


\begin{figure}\center
\includegraphics[width=\textwidth]{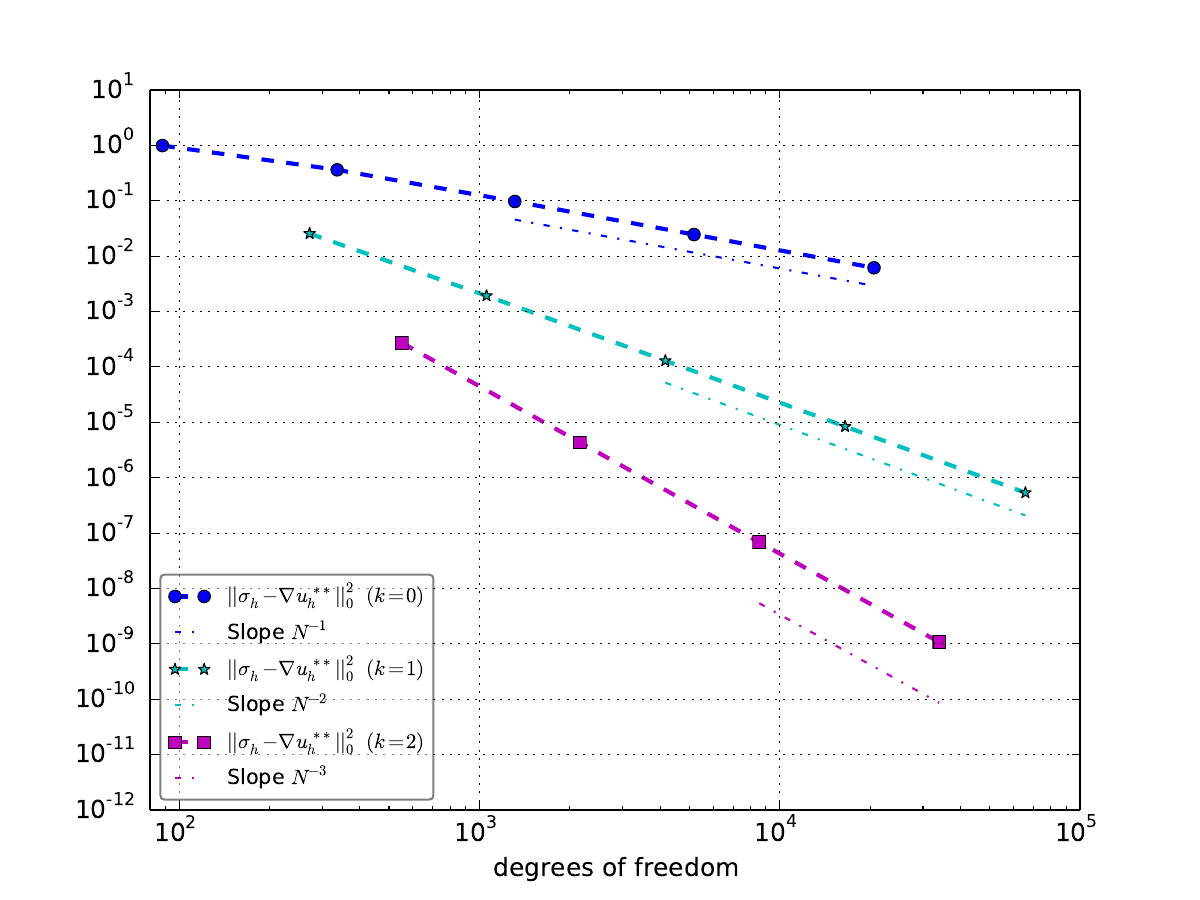}
\caption{Error estimator (square mesh).}
\label{fig:unieta}
\end{figure}

\begin{figure}\center
\begin{tabular}{l l | l l l l l l}
&$ k $& \\ \hline
$\vert\lambda - \lambda_h \vert$ & 0 &  1.002 &  1.010 &  1.007  &  1.004 \\
&1& 2.085&  2.016&  2.011&  2.006\\
&2&  3.018&  3.018&  3.011&  \\ \hline
$\|u-u_h\|_0^2$ & 0& 1.059&  1.025&  1.011&  1.005\\
&1&  2.017&  2.015&  2.009&  2.005\\
&2 & 3.025&  3.019&  3.011&  \\ \hline
$\| \boldsymbol \sigma_{h} - \nabla u^{**}_h \|_0^2$ 
& 0 & 0.753&  0.967&  0.999&  1.003\\
&1&  1.907&  1.968&  1.984&  1.992\\
&2&  3.032&  3.019&  3.011& \\
\end{tabular}
\caption{Convergence rates (uniformed refinement on square mesh).}
\label{fig:conv}
\end{figure}

\begin{figure}\center
\includegraphics[width=\textwidth]{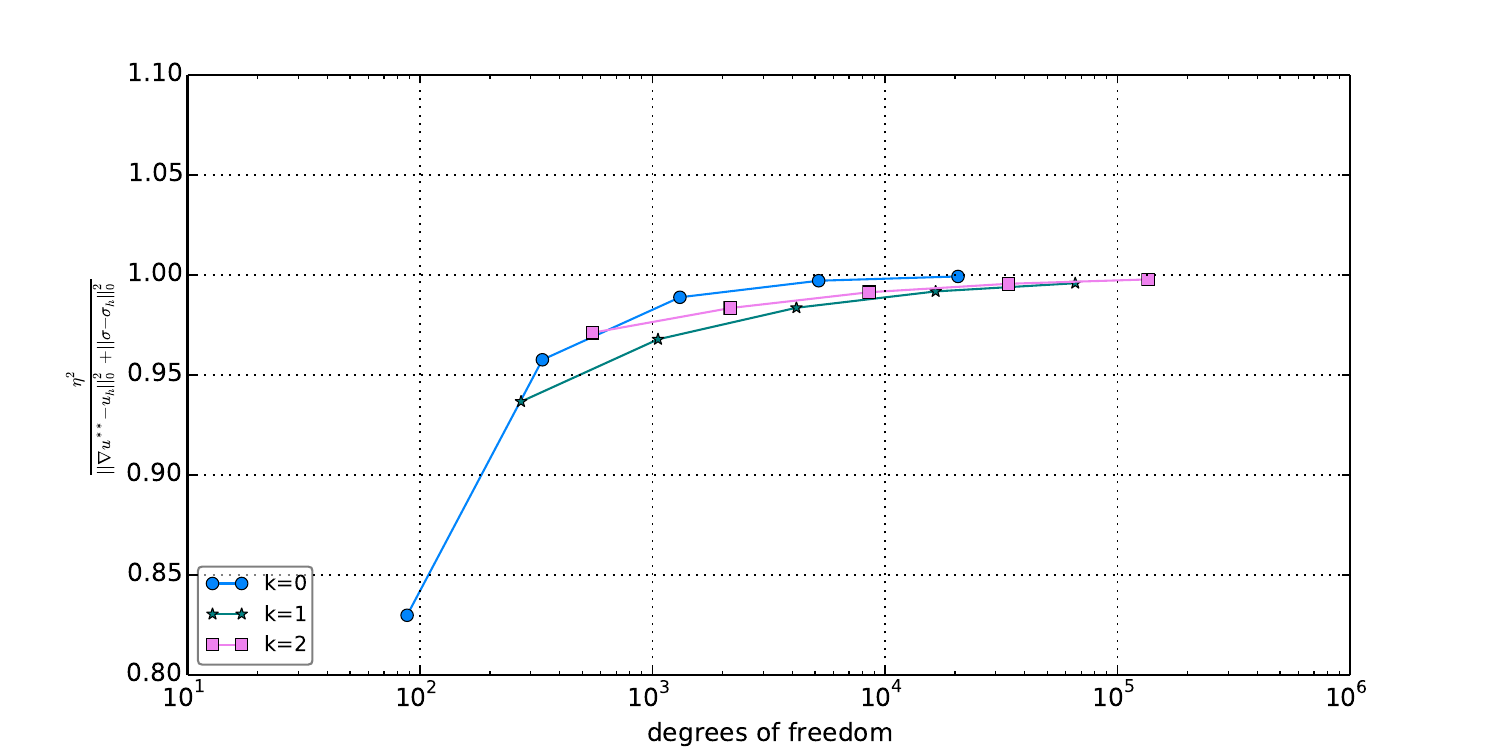}
\caption{Efficiency index (square mesh).}
\label{fig:unietac}
\end{figure}

\begin{figure}\center
\includegraphics[width=\textwidth]{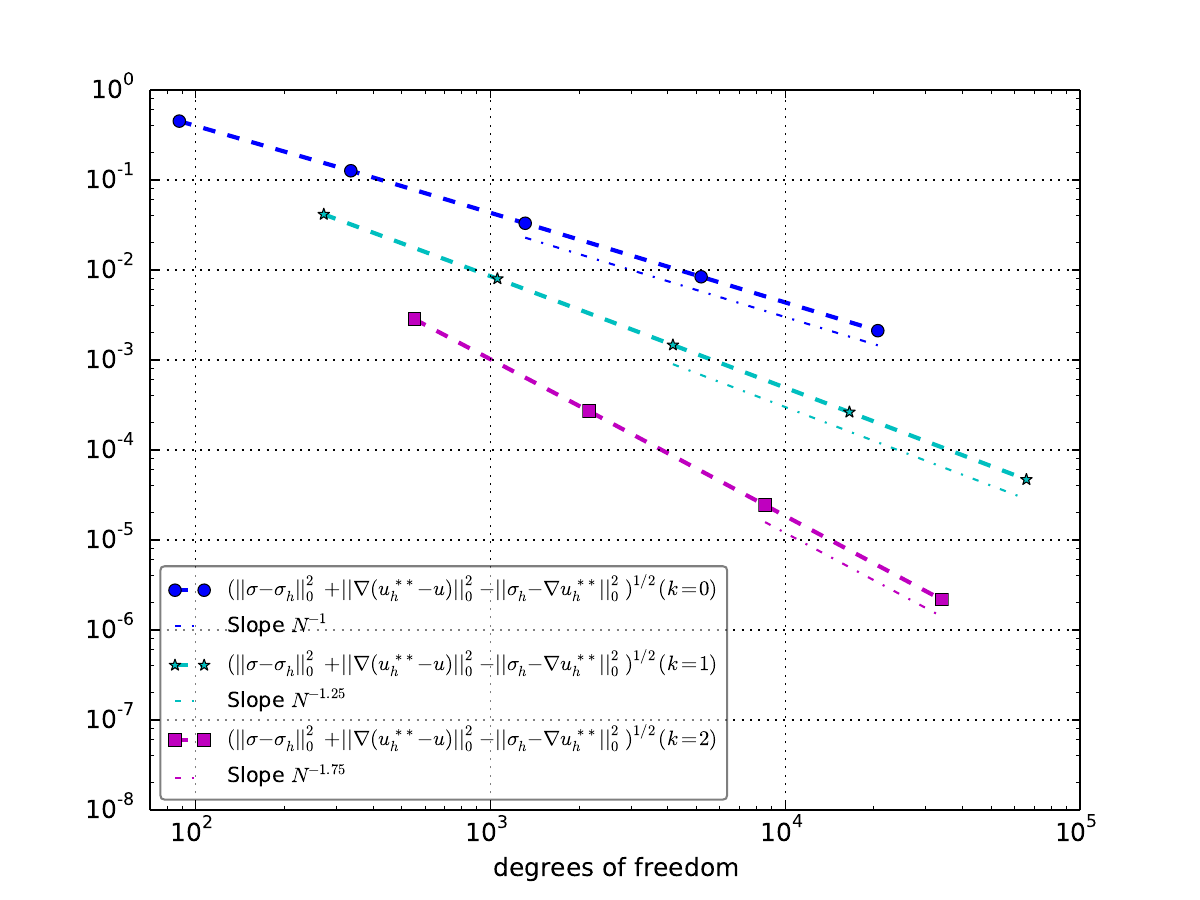}
\caption{Higher order terms in efficiency bound (square mesh).}
\label{fig:unihot}
\end{figure}

In order to illustrate the adaptive scheme, a second test case is concerned
with the L-shaped domain $[-1,1]^2 \backslash [0,1]^2$. Since the Sobolev
regularity is restricted by the biggest interior angle of the domain (see
classical regularity theory e.g. \cite{KozMazRos:97}), the lowest
eigenfunction can only be shown to be in $H^s$ for $s <4/3$ for our L-shaped
case. The convergence rate for the approximation of $\lambda$ for a uniform
refinement will be at best ${\dofs}^{-2/3}$, as illustrated in Figure
\ref{fig:adalambda}. There, we also see that the adaptive scheme is able to
retain the optimal convergence order. Similarly, the convergence rate for the
approximation of the eigenfunction for a uniform refinement will be at best
${\dofs}^{-1/3}$. In Figure \ref{fig:adaeta}, we see that the error estimator
converges with the optimal convergence rate. Figure \ref{fig:adamesh} shows
that the adaptive strategy refines the mesh where the singularity is expected,
i.e. in the reentrant corner.

\begin{figure}\center
\includegraphics[width=\textwidth]{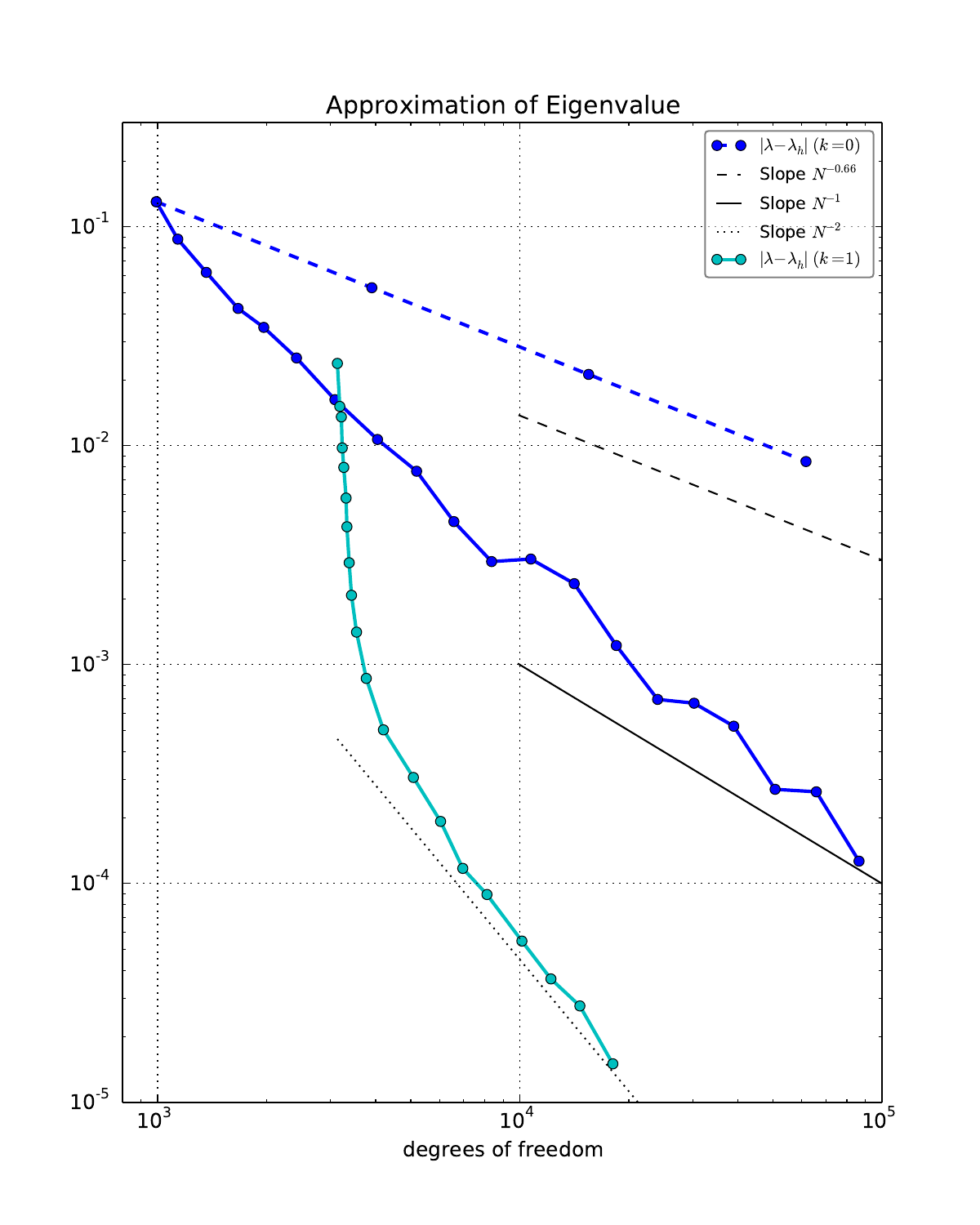}
\caption{Eigenvalue approximation on L-shaped domain.}
\label{fig:adalambda}
\end{figure}
\begin{figure}\center
\includegraphics[width=\textwidth]{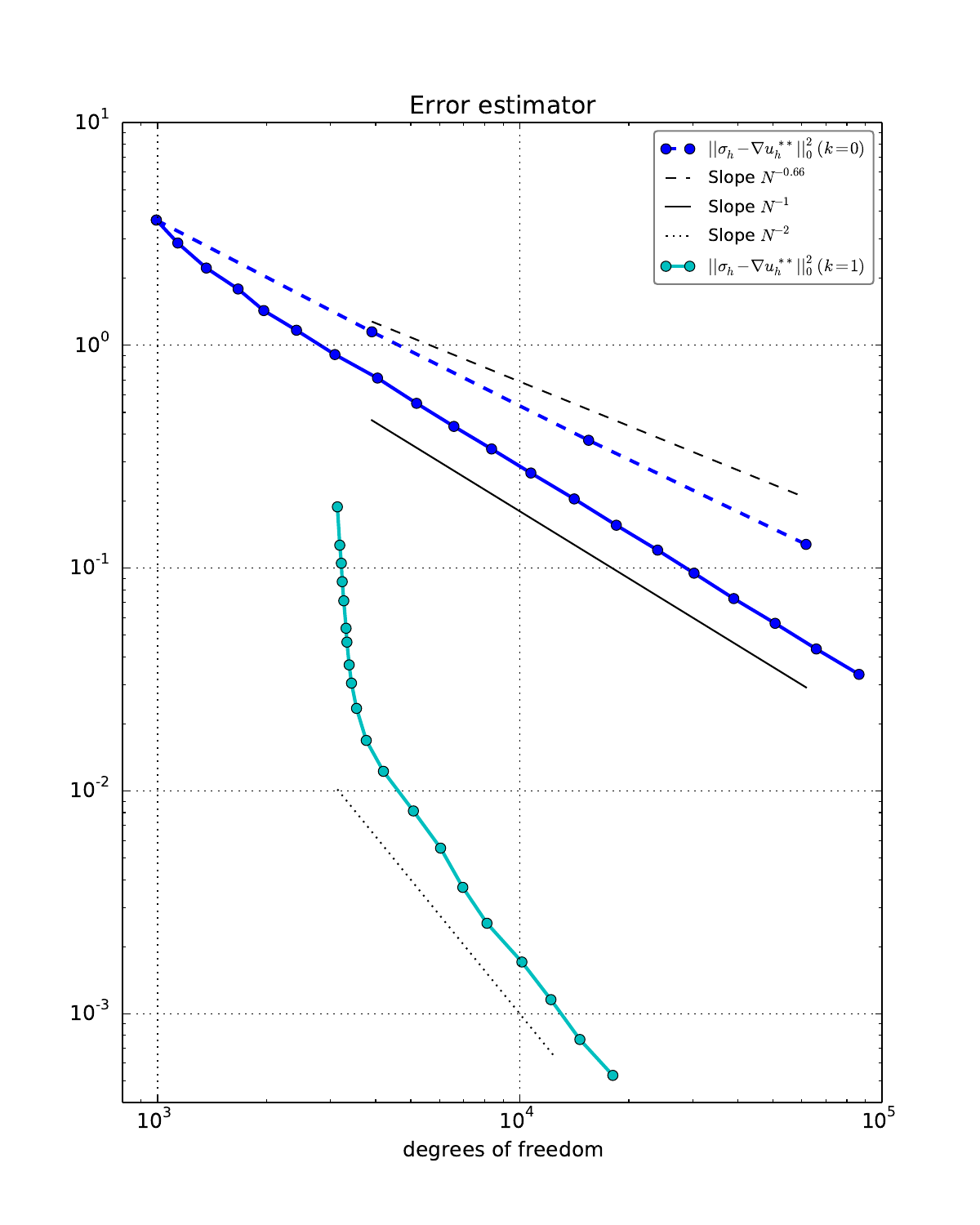}
\caption{Error estimator on L-shaped domain.}
\label{fig:adaeta}
\end{figure}
\begin{figure}\center
\includegraphics[width=0.3\textwidth, trim={10cm 4cm 10cm 4cm},clip]{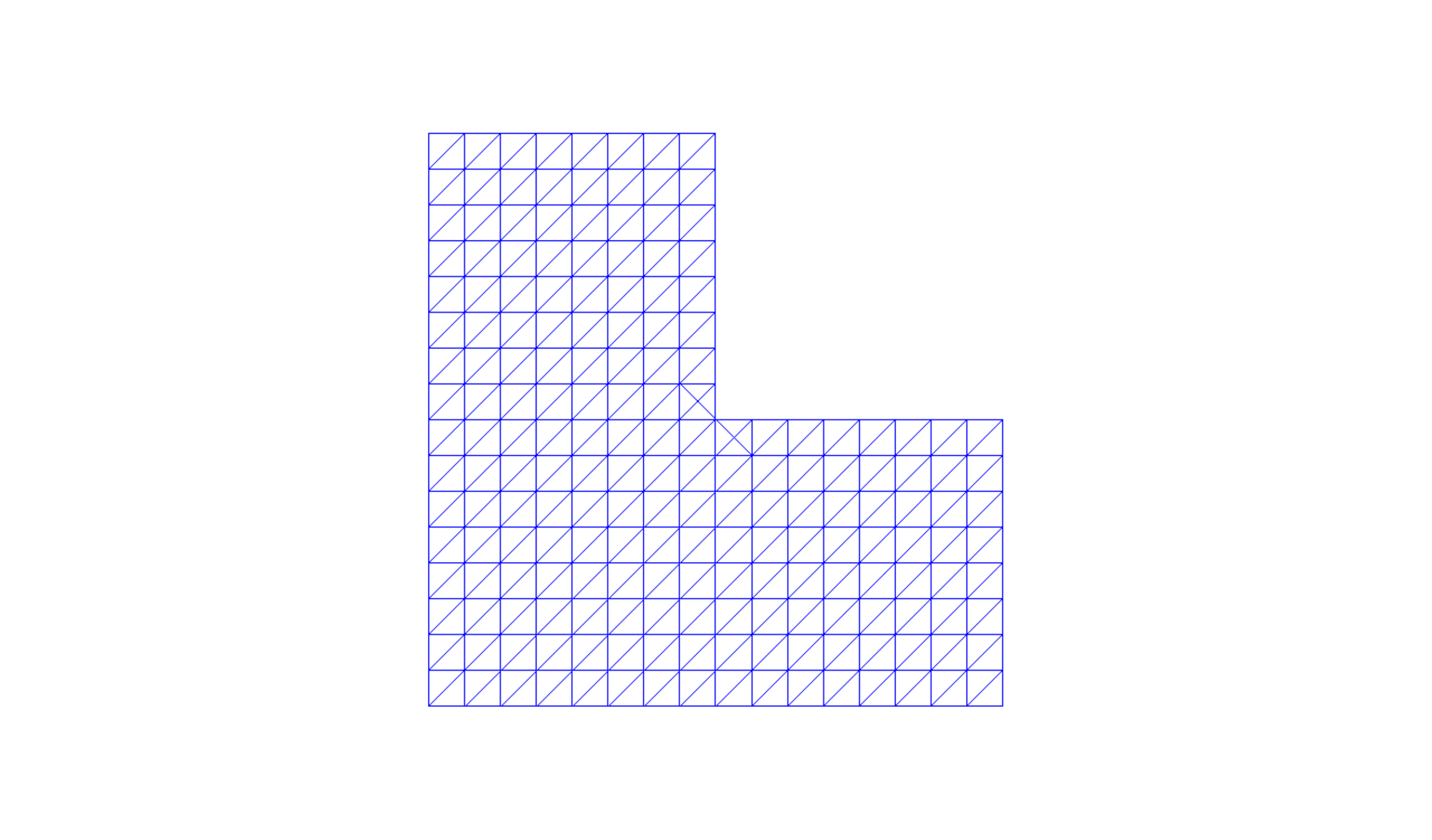}
\includegraphics[width=0.3\textwidth, trim={10cm 4cm 10cm 4cm},clip]{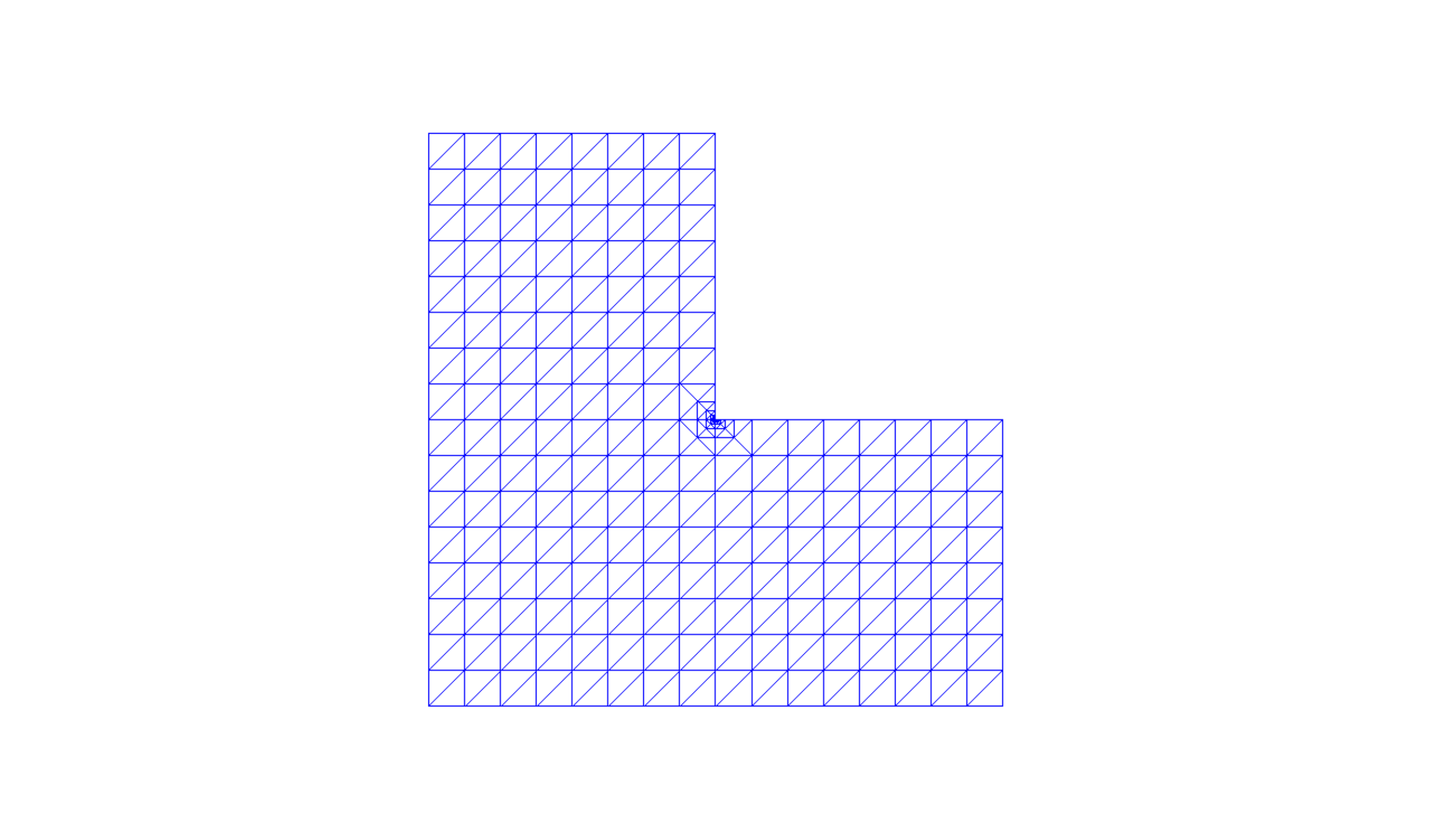}
\includegraphics[width=0.3\textwidth, trim={10cm 4cm 10cm 4cm},clip]{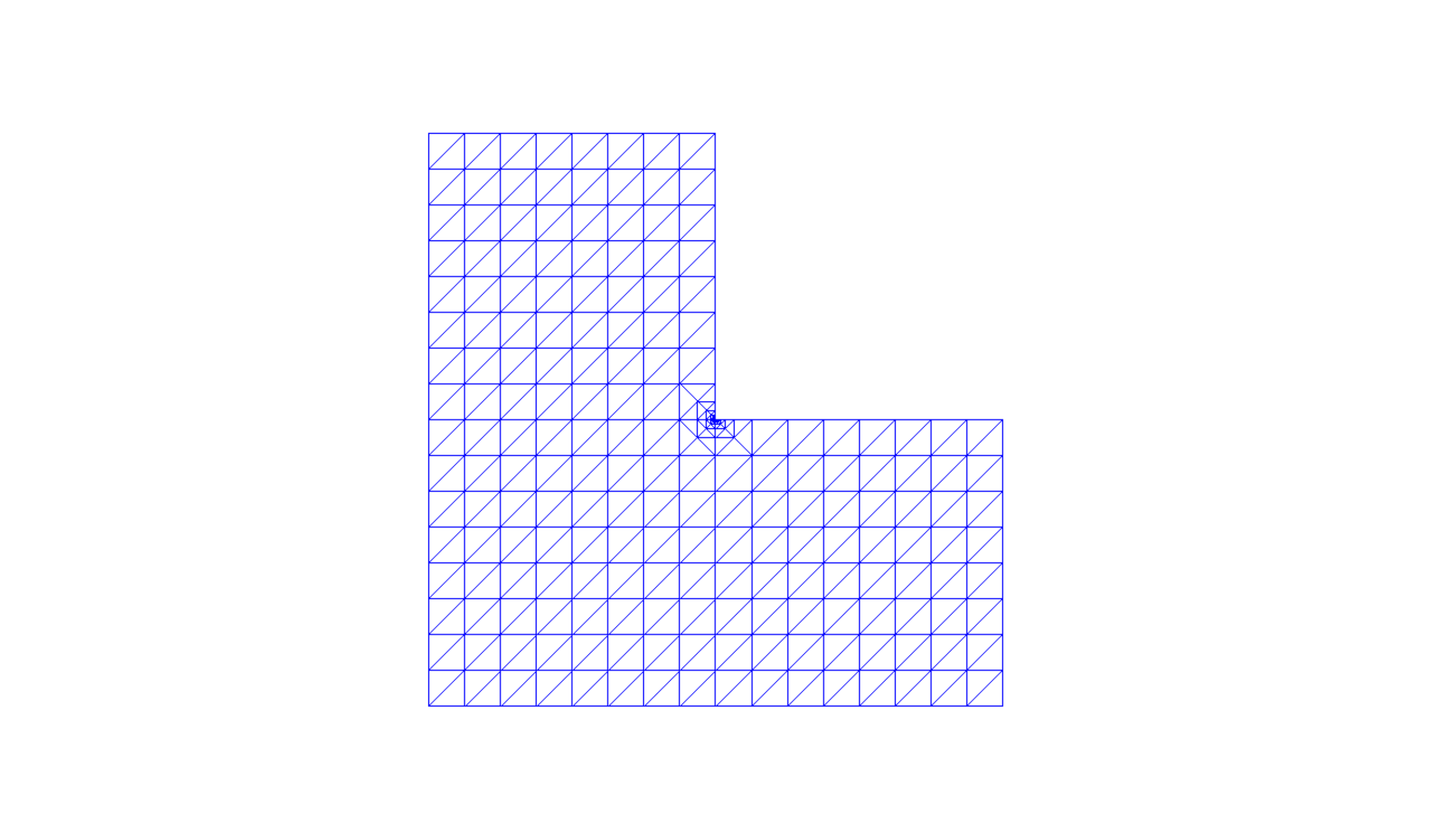}
\caption{Initial mesh, adaptive refinement after 19 steps for $k=0$ and $k=1$}
\label{fig:adamesh}
\end{figure}

\bibliographystyle{plain}
\bibliography{ref}

\end{document}